\documentclass[letterpaper, 11pt]{amsart}

\usepackage{amsmath,amssymb,amsxtra,amsthm,hyperref}
\usepackage{tikz}
\usepackage{multicol}
\usetikzlibrary{calc,decorations.markings}

\newcommand{\bh}[1] {\mathcal{B}(\mathcal{#1})}

\newcommand{\lspan} {\operatorname{span}}

\newcommand{\condref}[1] {\hyperref[cond.#1]{(#1)}}
\newcommand{\ZS}{Zappa-Sz\'{e}p }
\newcommand{\odometer}[1]{\mathbb{O}_{#1}}

\newcommand{\cH}{\mathcal{H}}

\newcommand{\cK}{\mathcal{K}}
\newcommand{\cL}{\mathcal{L}}

\newcommand{\cO}{\mathcal{O}}

\newtheorem{theorem}{Theorem}[section]

\newtheorem{corollary}[theorem]{Corollary}
\newtheorem{proposition}[theorem]{Proposition}   
 
\newtheorem{lemma}[theorem]{Lemma}

\theoremstyle{definition}
\newtheorem{definition}[theorem]{Definition}

\newtheorem{example}[theorem]{Example}
\newtheorem{remark}[theorem]{Remark}

\numberwithin{equation}{section}

\allowdisplaybreaks
\begin{document}

\title{Wold Decomposition on Odometer Semigroups}

\date{\today}


\begin{abstract} We establish a Wold-type decomposition for isometric and isometric Nica-covariant representations of the odometer semigroup. These generalize the Wold-type decomposition for commuting pairs of isometries due to Popovici and for pairs of doubly commuting isometries due to S\l oci\'{n}ski. 
\end{abstract}

\author{Boyu Li}
\address{Department of Mathematics and Statistics, University of Victoria, Victoria, B.C., Canada}
\email{boyuli@uvic.ca}
\date{\today}

\thanks{The author was supported by a fellowship of the Pacific Institute for the Mathematical Sciences. The author would also like to thank Adam Dor-on for many helpful discussions that lead to this project. The author would also like to thank the anonymous reviewer for many helpful comments.}

\maketitle

\section{Introduction}

In operator theory, the classical Wold decomposition theorem states that every isometric operator can be decomposed into a unitary component and a pure isometry component \cite[Theorem 1.1]{SFBook}. This powerful theorem soon becomes an essential tool in the study of operator theory and operator algebra. For example, the celebrated Coburn's theorem on C*-algebras generated by a proper isometry \cite{Coburn1967} is rooted in this result. 

There are numerous researches on generalizing this powerful result. Suciu first considered a Wold-type decomposition for a semigroup of isometries. He showed that a semigroup of isometries ${V_p}$ decomposes into three components \cite{Suciu1968}: a unitary component, a totally non-unitary component (corresponding to the unilateral shifts in the Wold decomposition), and lastly, a third component for which he called the ``strange" component (see also \cite{Popescu1998_Wold}). Fully characterizing these components is not an easy task, even for the seemingly simple case of a pair of commuting isometries. S\l oci\'{n}ski first obtained a Wold-type decomposition for a pair of doubly commuting isometries \cite[Theorem 3]{Slocinski1980}. He proved that a pair of doubly commuting isometries decomposes into four components, one for each of the four possible combinations in which each isometry is either unitary or pure. His result has been further generalized to product systems \cite{SZ2008} and higher dimensions in \cite{Sarkar2014}. However, without the doubly commuting assumption,  S\l oci\'{n}ski constructed an example where the pair of commuting isometries are neither unitary nor pure on a reducing subspace \cite[Example 1]{Slocinski1980}. This last mysterious piece was finally settled by Popovici \cite{Popovici2004} where he introduced the notion of a weak bi-shift. Along another direction, for families of non-commuting isometries, a well-known result of Popescu established a Wold-type decomposition for row isometries \cite{Popescu1989}. 

In this paper, we study the Wold-type decomposition for isometric representations of the odometer semigroup. The odometer semigroup, also known as the adding machine or the Baumslag-Solitar monoid $BS(1,n)^+$, encodes a simple yet intriguing semigroup structure. Its semigroup C*-algebra and boundary quotient has been studied extensively in recent years \cite{ART2018, BRRW, Spielberg2012}.

We consider two classes of representations. One being the isometric representation, which is a generalization for a pair of commuting isometries. Another being the isometric Nica-covariant representation, which is a generalization for a pair of doubly commuting isometries. 
We obtain a Wold-type decomposition for each class of representations, generalizing both Popovici's and S\l oci\'{n}ski's results in this context (Theorem \ref{thm.Wold} and Theorem \ref{thm.Wold.NC}).
Prior literature on Wold decomposition often assume the doubly commuting condition, which is a special case of the Nica-covariance condition.  However, our understanding of the Wold decomposition in the general Nica-covariance setting is limited. Our work contributes to this vein in the literature. In particular, our results on the odometer semigroups is the first Wold-type decomposition in the more general Nica-Covariant setting.
Finally, we provide several concrete atomic representations as examples. In particular, Example \ref{ex.ws} gives an example of weak bi-shift representation in our context, that bear a close resemblance to the example of S\l oci\'{n}ski.  

\section{Preliminary} 

We first recall that for an operator $T\in\bh{H}$, a subspace $\cK\subset \cH$ is invariant for $T$ if $T\cK\subset \cK$. We say $\cK$ reduces $T$ (equivalently, we say $\cK$ is a reducing subspace for $T$) if $\cK$ is invariant for both $T$ and $T^*$. On the Hilbert space $\ell^2(\mathbb{N})=\overline{\lspan}\{e_n: n\geq 0\}$, the unilateral shift is the isometry $S$ defined uniquely by mapping $e_n \mapsto e_{n+1}$. For an isometry $V\in\bh{H}$, a subspace $\cL$ is called wandering for $V$ (equivalently, we say $\cL$ is a wandering vector space for $V$) if $\{V^m \cL: m\geq 0\}$ are pairwise orthogonal. Given a wandering subspace $\cL$ for $V$, we can build a invariant subspace $\cK=\bigoplus_{m\geq 0} V^m \cL$ for $V$, on which $V$ is unitarily equivalent to a direct sum of $\dim \cL$-copies of the unilateral shift. This space $\cK$ is reducing for $V$ if $\cL$ is invariant under $V^*$. The Wold decomposition theorem states that every isometry can be decomposed into a unitary and a direct sum of unilateral shifts. 

\begin{theorem}[Wold] Let $V$ be an isometry on a Hilbert space $\cH$. Then $\cH$ can be decomposed as a direct sum of two reducing subspaces $\cH_u$ and $\cH_s$ for $V$, such that $V|_{\cH_u}$ is unitary and $V|_{\cH_s}$ is unitarily equivalent to a direct sum of unilateral shifts. Moreover, this decomposition is unique and we can describe $\cH_u$ and $\cH_s$ explicitly by: 
\begin{align*}
\cH_u &= \bigcap_{m\geq 0} V^m \cH; \\
\cH_s &= \bigoplus_{m\geq 0} V^m \ker V^*.
\end{align*}
Here, $\ker V^*$ is a wandering subspace for $V$.
\end{theorem}

We call an isometry $V$ \emph{pure} if $\cH_u=\{0\}$. Equivalently, a pure isometry is a direct sum of unilateral shifts. 
The spaces $\cH_u$ and $\cH_s$ are the largest reducing subspaces for $V$ on which $V$ is unitary and pure respectively. In other words, if $\cH_0$ reduces $V$ and $V|_{\cH_0}$ is unitary (or pure), then $\cH_0\subset \cH_u$ (or $\cH_0\subset \cH_s$). 

The Wold-type decomposition for a pair of commuting isometries is not simple. S\l oci\'{n}ski first obtained a Wold-type decomposition when two isometries $S_1, S_2$ doubly commute (that is, $S_1$ commutes with both $S_2$ and $S_2^*$) \cite[Theorem 3]{Slocinski1980}.

\begin{theorem}[S\l oci\'{n}ski] For a pair of doubly commuting isometries $S_1, S_2$ on $\cH$. The space $\cH$ uniquely decomposes into a direct sum of four reducing subspaces $\cH=\cH_{uu}\oplus \cH_{us} \oplus \cH_{su} \oplus \cH_{ss}$, such that the restriction of $S_1,S_2$ on each subspace is unitary-unitary, unitary-pure, pure-unitary, and pure-pure respectively.  
\end{theorem}

In fact, the $\cH_{ss}$-component is unitarily equivalent to a direct sum of bi-shifts on $\ell^2(\mathbb{N}^2)$. However, when $S_1,S_2$ do not doubly commute, it is possible to have a reducing subspace on which $S_i$ are neither unitary nor pure. One example was given by S\l oci\'{n}ski \cite[Example 1]{Slocinski1980}, where we take $\cH=\overline{\lspan}\{e_{i,j}: i,j\in\mathbb{Z}, i\geq 0 \text{ or } j\geq 0\}$ and $S_1 e_{i,j}=e_{i+1,j}$ and $S_2 e_{i,j}=e_{i,j+1}$. The characterization of this final puzzling piece was finally settled by Popovici, where he introduced the notion of weak bi-shift \cite[Definition 2.5]{Popovici2004}.

\begin{definition} A pair of commuting isometries $S_1, S_2$ is called a weak bi-shift if $S_1|_{\cap_{j\geq 0} \ker S_2^* S_1^j}$, $S_2|_{\cap_{j\geq 0} \ker S_1^* S_2^j}$, and $S_1S_2$ are pure isometries. 
\end{definition}

The final pure-pure component in S\l oci\'{n}ski's result is then replaced by a weak bi-shift component \cite[Theorem 2.8]{Popovici2004}.

\begin{theorem}[Popovici]\label{thm.Popovici} For a pair of commuting isometries $S_1, S_2$ on $\cH$. The space $\cH$ decomposes into a direct sum of four reducing subspaces $\cH=\cH_{uu}\oplus \cH_{us} \oplus \cH_{su} \oplus \cH_{ws}$, such that the restriction of $S_1,S_2$ on each subspace is unitary-unitary, unitary-pure, pure-unitary, and weak bi-shift respectively. Moreover, this decomposition is unique and we can explicitly write out the spaces by the following formulae
\begin{align*}
    \cH_{uu} &= \bigcap_{n\geq 0} (S_1S_2)^n \cH \\
    \cH_{us} &= \bigoplus_{n\geq 0} S_2^n \left(\bigcap_{m\geq 0} S_1^m \big(\cap_{i\geq 0} \ker S_2^* S_1^i\big)\right) \\
    \cH_{su} &= \bigoplus_{n\geq 0} S_1^n \left(\bigcap_{m\geq 0} S_2^m \big(\cap_{i\geq 0} \ker S_1^* S_2^i\big)\right)
\end{align*}
\end{theorem}

On a different direction, Popescu considered a Wold-type decomposition for $n$ non-commuting isometries. A family $\{V_1,\cdots,V_n\}$ of $n$ non-commuting isometries is called a row isometry if $\sum_{k=1}^n V_k V_k^* \leq I$, or equivalently, $\{V_1,\cdots,V_n\}$ have pairwise orthogonal ranges. 
It is called a row unitary if $\sum_{k=1}^n V_kV_k^* = I$ and such $\{V_1,\cdots,V_n\}$ are often called Cuntz isometries as they generate the Cuntz algebra $\cO_n$. Row isometries can be viewed as a representation of the free semigroup $\mathbb{F}_n^+$, where for each $\mu=\mu_1\cdots\mu_m\in\mathbb{F}_n^+$, $V_\mu=V_{\mu_1}\cdots V_{\mu_m}$. 
One way to build a row isometry is to consider the left regular representation of $\mathbb{F}_n^+$ on $\ell^2(\mathbb{F}_n^+)=\overline{\lspan}\{e_\mu\}$ and set $V_i e_\mu=e_{i\mu}$. 
A space $\cL$ is called wandering for $\{V_1,\cdots,V_n\}$ if $\{V_\mu \cL: \mu\in\mathbb{F}_n^+\}$ are pairwise orthogonal. Given a wandering subspace $\cL$, one can show that on the reducing subspace $\cK=\bigoplus_{\mu\in\mathbb{F}_n^+} V_\mu \cL$, $\{V_1,\cdots,V_n\}$ is unitary equivalent to a direct sum of $\dim \cL$-copies of left regular representations. Popescu showed that similar to the Wold decomposition of a single isometry, every row isometry decomposes into a row unitary component and a direct sum of left regular representations \cite[Theorem 1.3]{Popescu1989}. 

\begin{theorem}[Popescu] Let $\{V_1,\cdots,V_n\}$ be a family of $n$ isometries with orthogonal ranges on $\cH$. Then $\cH$ decomposes into two reducing subspaces $\cH=\cH_u\oplus \cH_s$, such that $\{V_1,\cdots,V_n\}$ is a row unitary on $\cH_u$ and is a  direct sum of left regular representations on  $\cH_s$. Moreover, this decomposition is unique and we can explicitly describe $\cH_u$ and $\cH_s$ by:
\begin{align*}
\cH_u &= \bigcap_{m\geq 0} \bigoplus_{|\mu|=m} V_\mu \cH; \\
\cH_s &= \bigoplus_{\mu\in\mathbb{F}_n^+} V_\mu  (\bigcap_{k=1}^n \ker V_k^*).
\end{align*}
Here, $\bigcap_{k=1}^n \ker V_k^*$ is a wandering subspace for $\{V_1,\cdots,V_n\}$. 
\end{theorem}

We say a row isometry $\{V_1,\cdots,V_n\}$ is \emph{pure} if $\cH_u=\{0\}$, in other words, if it is a direct sum of copies of the left regular representation. Similar to the Wold decomposition for a single isometry, the spaces $\cH_u$ and $\cH_s$ are the largest reducing subspaces for $\{V_i\}$ on which $\{V_i\}$ is unitary and pure respectively. In other words, if $\cH_0$ reduces $\{V_i\}$ and $\{V_i\}|_{\cH_0}$ is unitary (or pure), then $\cH_0\subset \cH_u$ (or $\cH_0\subset \cH_s$). 

In this paper, we study a Wold-type decomposition for two classes of representation of the odometer semigroup $\odometer{n}$. The odometer semigroup $\odometer{n}$ is generated by $n$ free generators $v_1,\cdots,v_n$ and one extra generator $w$ such that $wv_k=v_{k+1}$ for all $1\leq k\leq n-1$ and $wv_n=v_1w$. It is also known as the adding machine, where $w$ mimics the add-one operation on the free semigroup $\mathbb{F}_n^+$. One can view $\odometer{n}$ as a \ZS product of the free semigroup $\mathbb{F}_n^+$ by $\mathbb{N}$ \cite{BRRW}. The semigroup $\odometer{n}$ is also isomorphic to the Baumslag Solitar monoid $BS(1,n)^+$ which is generated by two generators $a,b$ with $b^n a=ab$. The isomorphism can be realized by identifying $b$ with $w$ and $b^{k-1}a$ as $v_k$. As a Baumslag-Solitar monoid, each element $x\in\odometer{n}$ has a unique representation \cite[Proposition 2.3]{Spielberg2012}:
\[x=w^{a_1-1} v_1 w^{a_2-1} v_1 \cdots w^{a_m-1} v_1 w^N, 1\leq a_i \leq n, N\geq 0.\]
This is obtained by shifting as many $w$ to the right as possible. Using the notation $w^{a_i-1} v_1=v_{a_i}$, we have that each $x\in\odometer{n}$ has a unique representation $x=v_\mu w^N$ for some $\mu\in\mathbb{F}_n^+$ and $N\geq 0$. One can also shift as many $w$ to the left as possible and show that every $x\in\odometer{n}$ has a unique representation $x=w^p v_1^q$ for some $p,q\geq 0$. 

An \emph{isometric representation of the odometer semigroup $\odometer{n}$} is defined as a collection of an isometry $W$ and a row isometry $\{V_1,\cdots,V_n\}$ such that $WV_k=V_{k+1}$ for each $1\leq k\leq n-1$ and $WV_n=V_1W$. One may equivalently describe it as a pair of isometries $W$ and $V_1$ such that $W^n V_1=V_1 W$ and $\{W^k V_1: 0\leq k\leq n-1\}$ having pairwise orthogonal ranges. It is clear that when $n=1$, this is simply a pair of commuting isometries. 

In recent years, there is a great advancement in our understanding of semigroup C*-algebras, started by the celebrated work of Nica on covariant representations of quasi-lattice ordered semigroups. It is known that the odometer semigroups are quasi-lattice ordered in the sense of Nica \cite{Spielberg2012}.  One may refer to \cite{ART2018, BRRW, Spielberg2012} for studies of its semigroup C*-algebras and \cite{Nica1992} for the basic definition of the Nica-covariance condition. Without diving into the rich literature of the Nica-covariance condition, we call an isometric representation of $\odometer{n}$ \emph{Nica-covariant} if $W^* V_1 = V_n W^*$. This precisely coincides with isometric Nica-covariant representations of $\odometer{n}$ \cite{BRRW}. Notice that when $n=1$, this corresponds to the case of a pair of doubly commuting isometries.

\section{Wold decomposition for isometric representations}

We first derive a Wold-type decomposition for isometric representations of $\odometer{n}$. The decomposition can be greatly simplified when the representation is Nica-covariant. Let $\{W,V_1,V_2,\cdots,V_n\}$ be an isometric representation of the odometer semigroup $\odometer{n}$ on some Hilbert space $\cH$. First, for the unitary-row unitary component, we start with the following observation.

\begin{lemma}\label{lm.obs1} An isometric representation $\{W,V_1,\cdots,V_n\}$ has a unitary $W$ and a row unitary $\{V_1,\cdots,V_n\}$ if and only if $\{V_2,\cdots, V_n, V_1W\}$ is a row unitary. 
\end{lemma}

\begin{proof} The forward direction is trivial. For the converse, we observe that
\[I=V_1WW^* V_1^* + \sum_{i=2}^n V_i V_i^* \leq \sum_{i=1}^n V_i V_i^* \leq I.\]
This implies that $\{V_1,\cdots,V_n\}$ is a row unitary, and that $V_1WW^*V_1^*=V_1V_1^*$ which implies $W$ is unitary. 
\end{proof}

We would like to point out that the set $\{V_2,\cdots, V_n, V_1W\}$ is precisely $\{WV_i: 1\leq i\leq n\}$. Recall that the unitary-unitary component for a pair of commuting isometries $S_1, S_2$ corresponds to the unitary component for their product $S_1S_2$ (\cite[Proposition 2.1]{Popovici2004}). 
We prove that the unitary-row unitary component $\cH_{uu}$ precisely corresponds to the row unitary component for $\{V_2,\cdots, V_n, V_1W\}$. 

\begin{proposition}\label{prop.uu} Let 
\[\cH_{uu}=\bigcap_{m\geq 0} \left(\bigoplus_{|\mu|=m} \prod_{j=1}^m WV_{\mu_j} \cH\right).\]
Then $\cH_{uu}$ reduces $W$ and $V_i$'s, $W|_{\cH_{uu}}$ is unitary, and $\{V_1, \cdots, V_n\}|_{\cH_{uu}}$ is a row unitary. Moreover, $\cH_{uu}$ is the largest subspace with this property. 
\end{proposition}

\begin{proof} We first show that $\cH_{uu}$ is reducing for both $W$ and $V_i$. For each $m\geq 0$, denote $\cL_m=\bigoplus_{|\mu|=m} (\prod_{j=1}^m WV_{\mu_j}) \cH$. Take any $|\mu|=m$ and $h\in\cH$, consider the vector $k=\prod_{j=1}^m WV_{\mu_j} h$. First consider $Wk$, we repeatedly use the fact that $WV_i=V_{i+1}$ if $i\neq n$ and $WV_n=V_1W$ and obtain:
\[Wk=\begin{cases}
(WV_{\mu_1+1}) \prod_{j=2}^m WV_{\mu_j} h, & \text{ if } \mu_1\neq n; \\
(WV_1) (WV_{\mu_2+1}) \prod_{j=3}^m WV_{\mu_j} h, & \text{ if } \mu_1=n, \mu_2\neq n; \\
& \vdots \\
(WV_1) (WV_1) \cdots (WV_1) W h, & \text{ if } \mu_1=\mu_2=\cdots=\mu_m=n. \\
\end{cases}
\]
In any case, we have that $Wk\in \cL_m$. Similarly, for $W^*k$, 
\[W^*k=W^*\prod_{j=1}^m WV_{\mu_j} h=V_{\mu_1} \prod_{j=2}^m WV_{\mu_j} h\] 
We repeatedly use the fact that $V_i=WV_{i-1}$ if $i\neq 1$ and $V_1W=WV_n$ and obtain:
\[W^* k=\begin{cases}
(WV_{\mu_1-1}) \prod_{j=2}^m WV_{\mu_j} h, & \text{ if } \mu_1\neq 1; \\
(WV_n) (WV_{\mu_2-1}) \prod_{j=3}^m WV_{\mu_j} h, & \text{ if } \mu_1=1, \mu_2\neq 1; \\
& \vdots \\
\underbrace{(WV_n) (WV_n) \cdots (WV_n)}_{m-1} V_1 h, & \text{ if } \mu_1=\mu_2=\cdots=\mu_m=1. \\
\end{cases}
\]
Therefore, $W^*k\in \cL_m$ with the exception of the last scenario in which $W^* k \in \cL_{m-1}$. This proves that $\cH_{uu}$ reduces $W$. 

Now consider $V_ik=V_i \prod_{j=1}^m WV_{\mu_j} h$. We have that
\[V_ik=\begin{cases}
(W V_{i-1}) \prod_{j=1}^m WV_{\mu_j} h, & \text{ if } i\neq 1; \\
(WV_n) (WV_{\mu_1-1}) \prod_{j=2}^m WV_{\mu_j} h, & \text{ if } i=1, \mu_1\neq 1; \\
& \vdots \\
(WV_n) (WV_n) \cdots (WV_n) V_1 h, & \text{ if } i=\mu_1=\mu_2=\cdots=\mu_m=1. \\
\end{cases}
\]
We have that $V_ik\in \cL_{m+1}$ with the exception of the last scenario in which $V_ik\in \cL_{m}$. Finally, 
\[V_i^*k=\begin{cases}
V_i^* V_{\mu_1+1} \prod_{j=2}^m WV_{\mu_j} h, &\text{ if } \mu_1\neq n;\\
V_i^* V_1 W \prod_{j=2}^m WV_{\mu_j} h, & \text{ if } \mu_1=n.
\end{cases} 
\]
Since the $V$'s have orthogonal ranges, $V_i^* k$ is either $0$ when $\mu_1\neq n$ or $W \prod_{j=2}^m WV_{\mu_j} h\in W\cL_{m-1}\subset \cL_{m-1}$ when $\mu_1=n$. As a result, $\cH_{uu}$ reduces each $V_i$.

By \cite[Theorem 1.3]{Popescu1989}, $\cH_{uu}$ is the largest reducing subspace for the family $\{V_2,V_3,\cdots,V_n,V_1W\}$ on which this family is a row unitary. By Lemma \ref{lm.obs1}, this implies that $W|_{\cH_{uu}}$ is a unitary and $\{V_1,\cdots,V_n\}|_{\cH_{uu}}$ is a row unitary. If $\cH_0$ is any subspace that reduces $W$ and $V_i$'s such that $W|_{\cH_{0}}$ is a unitary and $\{V_1,\cdots,V_n\}|_{\cH_{0}}$ is a row unitary, by Lemma \ref{lm.obs1}, we have that the family $\{V_2,V_3,\cdots,V_n,V_1W\}$  is a row unitary on $\cH_0$ and thus $\cH_0\subset \cH_{uu}$ by \cite[Theorem 1.3]{Popescu1989}. 
\end{proof}

\begin{remark} We would like to point out that the unitary-row unitary component is often of special interest because it gives rise to representations of the boundary quotient semigroup C*-algebra \cite{BRRW}. 
\end{remark}

\begin{proposition}\label{prop.us} Let \[\cH_{us}=\bigoplus_{\mu\in\mathbb{F}_n^+} V_\mu\left(\bigcap_{m\geq 0} W^m \Big(\bigcap_{i=1}^n \bigcap_{j\geq 0} \ker V_i^* W^j\Big)\right).\]
Then $\cH_{us}$ reduces $W$ and $V_i$'s, $W|_{\cH_{us}}$ is unitary, and $\{V_1,\cdots,V_n\}|_{\cH_{us}}$ is pure. Moreover, $\cH_{us}$ is the largest subspace with this property.
\end{proposition}

\begin{proof} Let $\cL=\bigcap_{m\geq 0} W^m \Big(\bigcap_{i=1}^n \bigcap_{j\geq 0} \ker V_i^* W^j\Big)$. It is clear that for each $1\leq i\leq n$, $V_i^* \cL=\{0\}$. Therefore, $\cL\subset \bigcap_{i=1}^n \ker V_i^*$. By \cite[Theorem 1.3]{Popescu1989}, $\cL$ is a wandering subspace for $\{V_1,\cdots,V_n\}$. Therefore, $\cH_{us}$ reduces each $V_i$ and $\{V_1,\cdots,V_n\}|_{\cH_{us}}$ is pure. 

To see it reduces $W$, first notice that $\cL\subset \bigcap_{m\geq 0} W^m \cH$ and thus by the Wold decomposition of a single isometry, $\cL$ reduces $W$ and $W|_\cL$ is unitary. Therefore, $\cL=W\cL$. For each $\mu\in\mathbb{F}_n^+$, we have
\[WV_\mu \cL = \begin{cases}
V_{\mu_1+1}V_{\mu_2}\cdots V_{\mu_m} \cL, & \text{ if } \mu_1\neq n; \\
V_1 V_{\mu_2+1}\cdots V_{\mu_m} \cL, & \text{ if } \mu_1=n, \mu_2\neq n; \\
& \vdots \\
V_1\cdots V_1 W\cL, & \text{ if } \mu_1=\mu_2=\cdots=\mu_m=n;
\end{cases}
\]
and,
\[W^*V_\mu \cL = \begin{cases}
V_{\mu_1-1}V_{\mu_2}\cdots V_{\mu_m} \cL, & \text{ if } \mu_1\neq 1; \\
V_n V_{\mu_2-1}\cdots V_{\mu_m} \cL, & \text{ if } \mu_1=1, \mu_2\neq 1; \\
& \vdots \\
V_n\cdots V_n \cL, & \text{ if } \mu_1=\mu_2=\cdots=\mu_m=1.
\end{cases}
\]
Here, in the last case, $W^* V_1^m \cL=W^* V_1^m W\cL=V_n^m \cL$. As a result, $\cH_{us}$ reduces $W$. It is also clear from the computation that for each $m\geq 0$, $W$ is unitary on $\bigoplus_{|\mu|=m} V_\mu \cL$.  Therefore, $W$ is unitary on $\cH_{us}$. 

Suppose now $\cH_0$ is another reducing subspace on which $W$ is unitary and $\{V_1,\cdots,V_n\}$ is pure. Apply \cite[Theorem 1.3]{Popescu1989}, we have that $\cH_0=\bigoplus_{\mu\in\mathbb{F}_n^+} V_\mu \cL_0$ where $\cL_0=\bigcap_{i=1}^n \ker(V_i^*)\cap \cH_0$ is a wandering subspace for $\{V_1,\cdots,V_n\}$ that generate $\cH_0$. We now prove that $\cL_0\subset \cL$. 

We first observe that, since $W$ is unitary on $\cH_0$, 
\[W^* (W V_n) W^* = W^* (V_1 W) W^*.\]
It implies that  on $\cH_0$, $V_n W^* = W^* V_1$ and thus $V_1^* W=W V_n^*$. In other words, the representation is in fact Nica-covariant on $\cH_0$. 

Let us prove that $\cL_0$ reduces $W$. Take any $h\in\cL_0$, we need to show that $Wh, W^*h\in \ker(V_i^*)$ for all $1\leq i\leq n$. We have:
\[
V_i^* W h = \begin{cases}
V_{i-1}^* h, \text{ if } i\neq 1;\\
W V_n^* h, \text{ if } i=1;
\end{cases}
\]
In either case, since $h\in \ker(V_i^*)$ for all $1\leq i\leq n$, we have that $V_i^* Wh=0$ for all $i$ as well. Furthermore, we have:
\[
V_i^* W^* h = \begin{cases}
V_{i+1}^* h, \text{ if } i\neq n;\\
W^* V_1^* h, \text{ if } i=n;
\end{cases}
\]
Again, we have that $V_i^* W^* h=0$ in either case. This proves that $\cL_0$ reduces $W$. 

Finally, let us prove that $\cL_0\subset \cL$. Pick any $h\in \cL_0$. For any $m\geq 0$, since $W$ is unitary on $\cH_0$, $h=W^m W^{*m} h$. From the definition of $\cL$, it suffices to show that $W^{*m}h \in \bigcap_{i=1}^n \bigcap_{j\geq 0} \ker V_i^* W^j$.
 Equivalently, it suffices to show that $W^{*m}h \in \ker V_i^* W^j$  for all $1\leq i\leq n$ and $j\geq 0$, which is also equivalent to showing that $W^j W^{*m} h \in \ker V_i^*$ for all $1\leq i\leq n$ and $j\geq 0$. Since $\cL_0$ reduces $W$, the vector $k=W^j W^{*m} h\in \cL_0$ for all $j\geq 0$ and $m\geq 0$. By the definition of $\cL_0$, $k\in \ker(V_i^*)$ for all $1\leq i\leq n$. Therefore, we can prove that $\cL_0\subset \cL$ and thus $\cH_0\subset \cH_{us}$, establishing the maximality of $\cH_{us}$. \end{proof}

\begin{remark}\label{rm.us1} From the proof of Proposition \ref{prop.us}, we can also conclude that for each $m\geq 0$, the $m$-th graded subspace,
\[\cH_{us}^{(m)}=\bigoplus_{|\mu|=m} V_\mu\left(\bigcap_{m\geq 0} W^m \Big(\bigcap_{i=1}^n \bigcap_{j\geq 0} \ker V_i^* W^j\Big)\right),\] reduces $W$. Moreover, $W$ is unitary on each of $\cH_{us}^{(m)}$. In fact, we show that $W$ is uniquely determined by $W|_{\cH_{us}^{(0)}}$. 
\end{remark}

\begin{corollary}\label{cor.us} Suppose $\{V_1,\cdots,V_n\}$ is a pure row isometry on $\cH$ that is generated by a wandering space $\cL$. Then each unitary operator $W_0\in \bh{L}$ determines a unique unitary operator $W$ on $\cH$ such that $W|_\cL=W_0$ and $\{W,V_1,\cdots,V_n\}$ is an isometric representation of the odometer semigroup $\odometer{n}$. Conversely, every unitary-pure row isometry representation arises in this manner. 
\end{corollary}

\begin{proof} Since $\cL$ is wandering for $V_i$, $\cH=\bigoplus_{\mu\in\mathbb{F}_n^+} V_\mu \cL$. For each $\mu=\mu_1\cdots\mu_m \in \mathbb{F}_n^+$ and each $h\in\cL$, define
\[W V_\mu h = \begin{cases}
V_{(\mu_1+1)\mu_2\cdots \mu_m} h , & \text{ if } \mu_1\neq n, \\
V_{1(\mu_2+1)\mu_3\cdots \mu_m} h, & \text{ if } \mu_1=n, \mu_2\neq n, \\
& \vdots \\
V_{\underbrace{1\cdots 1}_{m}} W_0 h, & \text{ if } \mu_1=\cdots = \mu_m=n. \\
\end{cases}
\]
In particular, for $m=0$, define $Wh=W_0h$ for all $h\in \cL$. One can easily see that for each $m$, $\bigoplus_{|\mu|=m} V_\mu \cL$ reduces $W$, and $W$ is unitary on this subspace. Moreover, from the construction, we also have $WV_k=V_{k+1}$ for all $1\leq k\leq n-1$ and $WV_n=V_1W$. Therefore, we have $\{W,V_1,\cdots,V_n\}$ is a unitary-pure row isometry representation of $\odometer{n}$. One may notice that the definition of $W$ on each vector $V_\mu h$ is forced by the relations on $\odometer{n}$, and thus $W$ is unique. 

Conversely, if $\{W,V_1,\cdots,V_n\}$ is a unitary-pure row isometry representation of $\odometer{n}$, then $\cH=\cH_{us}=\bigoplus_{\mu\in\mathbb{F}_n^+} V_\mu \cL$ for some wandering subspace $\cL$ of $\{V_1,\cdots,V_n\}$. Define $W_0=W|_\cL$, and as observed in Remark \ref{rm.us1}, $W_0$ is unitary since $\cL$ reduces $W$. One can easily verify that we can recover $W$ from $W_0$.
\end{proof}

\begin{proposition}\label{prop.su} Let \[\cH_{su}=\bigoplus_{k\geq 0} W^k \left(\bigcap_{m\geq 0} V_1^m \Big(\bigcap_{j\geq 0} \ker W^* V_1^j\Big)\right).\]
Then $\cH_{su}$ reduces $W$ and $V_i$'s, $W|_{\cH_{su}}$ is pure and $\{V_1,\cdots,V_n\}|_{\cH_{su}}$ is a row unitary. Moreover, $\cH_{su}$ is the largest subspace  with this property.
\end{proposition}

\begin{proof} Let $\cL=\bigcap_{m\geq 0} V_1^m \Big(\bigcap_{j\geq 0} \ker W^* V_1^j\Big)$.  It is clear that $W^* \cL=\{0\}$ and thus $\cL\subset \ker W^*$. By the Wold decomposition of a single isometry, $\cL$ is a wandering subspace for $W$, and thus $\cH_{su}$ reduces $W$ and $W|_{\cH_{su}}$ is pure. 

Next, we show that $\cL$ reduces $V_1$. First, $\bigcap_{j\geq 0} \ker W^* V_1^j$ is clearly invariant under $V_1$ and thus $\cL$ is invariant under $V_1$. For any $h\in \cL$ and any $m\geq 0$, there exists $x\in \bigcap_{j\geq 0} \ker W^* V_1^j$ such that $h=V_1^{m+1} x$. Therefore, $V_1^* h=V_1^m x\in V_1^m \big(\bigcap_{j\geq 0} \ker W^* V_1^j\big)$. This proves that $\cL$ reduces $V_1$.

To see it reduces $V_i$, pick any $k\geq 0$, $h\in \cL$, and consider $W^k h\in \cH_{su}$. There exists an $m\geq 0$ such that $W^k V_1^m = V_\mu$ for some $|\mu|=m$. By the definition of $\cL$, there exists $x\in \bigcap_{j\geq 0} \ker W^* V_1^j$ such that $h=V_1^m x$. 
We have
\[V_i W^k h = V_i V_\mu x = W^{k'} V_1^{m+1} x = W^{m'} V_1 h.\]
Since $\cL$ is invariant for $V_1$, we have $W^{k'} V_1 h\in W^{k'} \cL\subset \cH_{su}$. In addition, we have $V_i^* W^k h = V_i^* V_\mu x$. This is either $0$ when $\mu_1\neq i$ or $V_{\mu_2}\cdots V_{\mu_m} x = W^{k''} V_1^{m-1} x = W^{k''} V_1^* h$. Again, since $\cL$ is invariant for $V_1^*$, $W^{k''} V_1^* h\in W^{k''}\cL \subset \cH_{su}$. Therefore, $\cH_{su}$ reduces $V_i$. Following the computation above can also easily establish that $\sum_{i=1}^n V_i V_i^*=I$ on $\cH_{su}$ because each $W^k h= V_\mu x$ is in the range of some $V_i$. 

Suppose now $\cH_0$ is another reducing subspace on which $W$ is a pure isometry and $\{V_1,\cdots,V_n\}$ is a  row unitary. By the Wold decomposition for a single isometry, $\cH_0=\oplus_{k\geq 0} W^k \cL_0$ where $\cL_0=\ker W^* \cap \cH_0$. It suffices to show that $\cL_0\subset \cL$. First, since $\{V_1,\cdots,V_n\}$ is a row unitary on $\cH_0$, when restricted on $\cH_0$, we have that $\sum_{i=1}^m V_i V_i^*=I$. Furthermore, on $\cH_0$, we have that:
\[I=\sum_{i=1}^m V_i (\sum_{j=1}^m V_j V_j^*) V_i^*=\sum_{i,j} V_{ij} V_{ij}^*.\]
Inductively, we have that for each $m\geq 1$, on $\cH_0$, 
\[\sum_{|\mu|=m} V_\mu V_\mu^*=I.\]
In particular, for any $h\in \ker W^*\cap \cH_0$ and $m\geq 1$, 
$h=\sum_{|\mu|=m} V_\mu V_\mu^* h$. Each $V_\mu = W^k V_1^m$ and $V_\mu V_\mu^* h=0$, except when $\mu=1\cdots 1$ and $V_\mu=V_1^m$. Therefore, $h=V_1^m V_1^{*m} h$ for all $m\geq 1$. It now suffices to show that $V_1^{*m} h\in \bigcap_{j\geq 0} \ker W^* V_1^j$ for all $m\geq 0$. Equivalently, it suffices to show that $W^* V_1^j V_1^{*m} h=0$ for all $m\geq 0$ and $j\geq 0$. Since $V_1^m V_1^{*m}h=h$ for all $m\geq 0$, we have that $V_1^j V_1^{*m}h=V_1^{j-m}h$ when $j\geq m$ and $V_1^j V_1^{*m}h=V_1^{*(m-j)} V_1^m V_1^{*m}h=V_1^{*(m-j)}h$ when $j<m$. It suffices to show that $W^* V_1^k h=0$ and $W^* V_1^{*k} h=0$ for all $k\geq 0$. We have,
\begin{align*}
 &  \mbox{  } W^* V_1^k h=0 \text{ and } W^* V_1^{*k} h=0 \text{ for all }k\geq 0, \\
\Longleftrightarrow &  \mbox{  } V_1^k h, V_1^{*k}h \in \ker W^* \text{ for all }k\geq 0, \\
\Longleftrightarrow &  \mbox{  } \ker W^* \cap \cH_0 \text{ reduces } V_1,\\
\Longleftrightarrow &  \mbox{  } \operatorname{ran} W \cap \cH_0 \text{ reduces } V_1.
\end{align*}
Pick $h\in\cH_0$ and $Wh\in \operatorname{ran} W \cap \cH_0=\operatorname{ran} W|_{\cH_0}$, we have that $V_1 Wh=WV_n h \in \operatorname{ran} W|_{\cH_0}$. Since $\{V_1,\cdots,V_n\}$ is a row unitary on $\cH_0$, we have $h=\sum_{i=1}^n V_i V_i^* h$. Therefore, 
\begin{align*}
    V_1^* Wh &= V_1^* W \sum_{i=1}^n V_i V_i^* h \\
    & = V_1^* (V_2V_1^* + V_3 V_2^* + \cdots + V_1 W V_n^*) h \\
    &= W V_n^* h \in  \operatorname{ran} W|_{\cH_0}. 
\end{align*}
Therefore, $\operatorname{ran} W|_{\cH_0}$ reduces $V_1$ and thus $\cL_0 \subset \cL$ and $\cH_0\subset \cH_{su}$, establishing the maximality of $\cH_{su}$. 
\end{proof}

\begin{remark} One may notice the space $\cH_{su}$ has seemingly the same definition as the $\cH_{su}$-component in Popovici's result for the pair $(W,V_1)$. However, the fundamental difference is that $W$ and $V_1$ satisfies $W^n V_1=V_1W$ and they do not commute unless $n=1$. Therefore, one cannot apply Popovici's result directly here.  
\end{remark}

Finally, we extend Popovici's notion of weak bi-shift to our context.

\begin{definition}\label{defn.ws} We say an isometric representation $\{W,V_1,\cdots,V_n\}$ is a \emph{weak bi-shift} if the operators $W|_{\bigcap_{i=1}^n \bigcap_{j\geq 0} \ker V_i^* W^j}$ and $V_1|_{\bigcap_{j\geq 0} \ker W^* V_1^j}$ are pure isometries and the family $\{V_2,\cdots,V_n,V_1W\}$ is a pure row isometry.
\end{definition}

\begin{theorem}\label{thm.Wold} Let $\{W,V_1,V_2,\cdots,V_n\}$ be an isometric representation of the odometer semigroup $\odometer{n}$ on some Hilbert space $\cH$. Then $\cH$ decomposes uniquely as
\[\cH=\cH_{uu}\oplus \cH_{us} \oplus \cH_{su} \oplus \cH_{ws},\]
such that,
\begin{enumerate}
    \item The subspaces $\cH_{uu}, \cH_{us}, \cH_{su}, \cH_{ws}$ are reducing for $W$ and $V_i$.
    \item The $W|_{\cH_{uu}}$ is unitary and $\{V_1,\cdots,V_n\}|_{\cH_{uu}}$ is a row unitary.
    \item The $W|_{\cH_{us}}$ is unitary and $\{V_1,\cdots,V_n\}|_{\cH_{us}}$ is pure.
    \item The $W|_{\cH_{su}}$ is pure, and $\{V_1,\cdots,V_n\}|_{\cH_{su}}$ is a row unitary.
    \item The family $\{W|_{\cH_{ws}}, V_1|_{\cH_{ws}}, \cdots, V_n|_{\cH_{ws}}\}$ is a weak bi-shift.
\end{enumerate}
\end{theorem}

\begin{proof} Define $\cH_{uu}$, $\cH_{us}$, and $\cH_{su}$ as in Proposition \ref{prop.uu}, Proposition \ref{prop.us}, and Proposition \ref{prop.su} respectively. Each is a reducing subspace such that conditions (2) through (4) hold. Let $\cH_{ws}$ be the orthogonal complement to $\cH_{uu}\oplus \cH_{us} \oplus \cH_{su}$, which must be reducing, proving condition (1). Since $\cH_{uu}$ is the largest subspace such that the family $\{V_2, \cdots, V_n, V_1W\}$ is a row unitary, this family is pure on $\cH_{uu}^\perp\supset \cH_{ws}$. By $\cH_{ws}\perp \cH_{us}$, we have that \[\cH_{ws}\perp \bigcap_{m\geq 0} W^m \Big(\bigcap_{i=1}^n \bigcap_{j\geq 0} \ker V_i^* W^j\Big).\]
The space $\bigcap_{i=1}^n \bigcap_{j\geq 0} \ker V_i^* W^j$ is clearly invariant for $W$ and consider the Wold decomposition for the single isometry $W|_{\bigcap_{i=1}^n \bigcap_{j\geq 0} \ker V_i^* W^j}$, its unitary component corresponds to $\bigcap_{m\geq 0} W^m \Big(\bigcap_{i=1}^n \bigcap_{j\geq 0} \ker V_i^* W^j\Big)$ which is orthogonal to $\cH_{ws}$. Therefore, $W|_{\bigcap_{i=1}^n \bigcap_{j\geq 0} \ker V_i^* W^j}$ is pure on $\cH_{ws}$. Similarly, since $\cH_{ws}\perp \cH_{su}$, we have 
\[\cH_{ws}\perp \bigcap_{m\geq 0} V_1^m \Big(\bigcap_{j\geq 0} \ker W^* V_1^j\Big).\]
This implies that $V_1|_{\bigcap_{j\geq 0} \ker W^* V_1^j}$ is pure. By Definition \ref{defn.ws}, the family $\{W|_{\cH_{ws}}, V_1|_{\cH_{ws}}, \cdots, V_n|_{\cH_{ws}}\}$ is a weak bi-shift, proving (5). The uniqueness of this decomposition can be easily established since $\cH_{uu}$, $\cH_{us}$, and $\cH_{su}$ are maximal for their respective properties and the subspace corresponding to the weak bi-shift component has to be orthogonal to $\cH_{uu}$, $\cH_{us}$, and $\cH_{su}$. 
\end{proof}

\section{Wold decomposition for Nica-covariant representations}

Now, let us focus on the case when $\{W,V_1,\cdots,V_n\}$ is isometric Nica-covariant, that is with the additional assumption that $W^* V_1=V_n W^*$. We first prove that an isometric representation of $\odometer{n}$ is automatically Nica-covariant if either $W$ is unitary or $\{V_1,\cdots,V_n\}$ is a row unitary. 

\begin{lemma}\label{lm.obs0} Let $\{W,V_1,\cdots,V_n\}$ be an isometric representation  of $\odometer{n}$. If $W$ is a unitary or $\{V_1,\cdots,V_n\}$ is a row unitary, then it is also a Nica-covariant representation. 
\end{lemma}

\begin{proof} When $W$ is unitary, 
\[W^* V_1 = W^* (V_1 W) W^* = W^* W V_n W^* = V_n W^*.\]
When $\{V_1,\cdots,V_n\}$ is a row unitary, 
\begin{align*}
    W^* V_1 &= (\sum_{k=1}^n V_k V_k^*) W^* V_1 
    = (\sum_{k=1}^n V_k (WV_k)^*) V_1  \\
    &= (V_1V_2^* + \cdots + V_{n-1}V_n^* + V_n W^* V_1^*) V_1 = V_n W^*. \qedhere
\end{align*}
\end{proof}

As a result, among the four components in Theorem \ref{thm.Wold}, only the weak bi-shift component may not be Nica-covariant. We would like to prove the only Nica-covariant weak bi-shift is a direct sum of the left-regular representations of $\odometer{n}$. 

Consider the usual Wold decomposition for the isometry $W$: let $\cH_u^W=\bigcap_{m\geq 0} W^m \cH$ and $\cH_s^W=\bigoplus_{m\geq 0} W^m \ker W^*$. 

\begin{lemma} When $\{W,V_1,\cdots,V_n\}$ is Nica-covariant, both $\cH_u^W$ and $\cH_s^W$ reduce all the $V_1,\cdots,V_n$.
\end{lemma}

\begin{proof} First of all, since $\cH=\cH_u^W\oplus \cH_s^W$, it suffices to show that $\cH_u^W$ reduces all the $V_1,\cdots, V_n$. We first prove that $\cH_u^W$ reduces $V_1$. Take any $h\in \cH^W_u$. Since $\cH_u^W=\bigcap_{k\geq 0} W^k \cH$, $h$ is in the range of each $W^k$. For each $k\geq 0$, we can write $h=W^k h_k$ for some $h_k\in\cH_u^W$. We have
\[V_1 h = V_1 W^k h_k = W^{nk} V_1 h_k \in \bigcap_{m=0}^{nk} W^m\cH.\]
On the other hand, the Nica-covariance condition implies that $V_1^* W = W V_n^* = W V_1^* W^{*(n-1)}$, and thus $V_1^* W^n=W V_1^*$. Pick $k=mn$, we have
\[V_1^* h = V_1^* W^{mn} h_{mn} =  W^m V_1^* h_{mn} \in  W^m\cH. \]
Therefore, $\cH_u^W$ reduces $V_1$. Now since $V_i = W^{i-1} V_1$, we have that 
\begin{align*}
V_i \cH_u^W &= W^{i-1} V_1 \cH_u^W \subset W^{i-1} \cH_u^W \subset \cH_u^W, \\
V_i^* \cH_u^W &= V_1^* W^{(i-1)*} \cH_u^W \subset V_1^* \cH_u^W \subset \cH_u^W.
\end{align*}
Therefore, $\cH_u^W$ reduces all the $V_i$ as well. 
\end{proof}

Since both $\cH_u^W$ and $\cH_s^W$ reduce $\{V_1,\cdots,V_n\}$, we can apply Popescu's Wold decomposition for row isometries on these these two reducing subspaces. As a result, we obtain:
\begin{enumerate}
    \item $\cH_{uu}^0 = \bigcap_{m\geq 0} \bigoplus_{|\mu|=m} V_\mu \cH_u^W$, on which $W$ is unitary and $\{V_1,\cdots,V_n\}$ is a row unitary;
    \item  $\cH_{us}^0 = \bigoplus_{\mu\in\mathbb{F}_n^+} V_\mu (\bigcap_{i=1}^n \ker V_i^* \cap \cH_u^W)$ and on which $W$ is unitary and $\{V_1,\cdots,V_n\}$ is pure;
    \item $\cH_{su}^0 = \bigcap_{m\geq 0} \bigoplus_{|\mu|=m} V_\mu \cH_s^W$, on which $W$ is pure and $\{V_1,\cdots,V_n\}$ is a row unitary;
    \item  $\cH_{ss}^0 = \bigoplus_{\mu\in\mathbb{F}_n^+} V_\mu (\bigcap_{i=1}^n \ker V_i^* \cap \cH_s^W)$ and on which $W$ is pure and  $\{V_1,\cdots,V_n\}$ is pure.
\end{enumerate}

By the uniqueness of the decomposition in Theorem \ref{thm.Wold}, we have that $\cH_{uu}^0=\cH_{uu}$, $\cH_{us}^0=\cH_{us}$, $\cH_{su}^0=\cH_{su}$, and finally $\cH_{ss}^0=\cH_{ws}$. As a result, both $W$ and $\{V_1,\cdots,V_n\}$ are pure if it is a Nica-covariant weak bi-shift. 
We further claim that the Nica-covariant weak bi-shift component  must be a direct sum of the left regular representation of the $\odometer{n}$. We call a space $\cL$ a wandering subspace for $\{W,V_1,\cdots,V_n\}$ if the collections $\{V_\mu W^m \cL: \mu\in\mathbb{F}_n^+, m\geq 0\}$ are pairwise orthogonal. It is easy to see that when $\cL$ is a wandering subspace, we can define $\cK=\bigoplus_{\mu\in\mathbb{F}_n^+} \bigoplus_{m\geq 0} V_\mu W^m \cL$ and it is easy to verify that on $\cK$, $\{W,V_1,\cdots,V_n\}$ is unitarily equivalent to a direct sum of $\dim \cL$-copies of the left regular representation of $\odometer{n}$.

\begin{proposition}\label{prop.NC1} Let $\{W,V_1,\cdots,V_n\}$ be a Nica-covariant weak bi-shift on $\cH$. Then $\cL=\bigcap_{i=1}^n \ker V_i^* \cap \ker W^*$ reduces $\{W, V_1,\cdots,V_n\}$, and we have $\cH=\bigoplus_{\mu\in\mathbb{F}_n^+} \bigoplus_{m\geq 0} V_\mu W^m \cL$ so that $\{W,V_1,\cdots,V_n\}$ is unitarily equivalent to a direct sum of $\dim \cL$-copies of the left regular representation of $\odometer{n}$.
\end{proposition}
\begin{proof} We have shown that when $\{W,V_1,\cdots,V_n\}$ is Nica-covariant, both $W$ and $\{V_1,\cdots,V_n\}$ are pure. Therefore, by the Wold decomposition for row isometries, let $\cL_0=\bigcap_{i=1}^n \ker V_i^*$, we have $\cH=\bigoplus_{\mu\in\mathbb{F}_n^+} V_\mu \cL_0$. 

We observe that $\cL_0$ reduces $W$: pick any $h\in \cL_0$ and any $1\leq i\leq n$, compute
\begin{align*}
    V_i^* W h &= \begin{cases} V_{i-1}^* h, & \text{ if } i\neq 1, \\
    W V_n^* h, & \text{ if } i=1.
    \end{cases}\\
    V_i^* W^* h &= \begin{cases} V_{i+1}^* h, & \text{ if } i\neq n, \\
    W^* V_1^* h, & \text{ if } i=n.
    \end{cases}
\end{align*}
In any case, since $V_i^* h=0$ for all $i$, $Wh, W^*h\in\cL_0$. 

Next, for each $1\leq i\leq n$ and $j\geq 0$, consider the space $\ker V_i^* W^j$. Repeatedly apply $V_1^* W = W V_n^*$ and $V_k^* = V_1^* W^{*(k-1)}$, we have $V_i^* W^j = W^{j'} V_{i'}^*$ for some $j'\geq 0$ and $1\leq i'\leq n$. Since $W$ is an isometry, $\ker V_i^* W^j = \ker W^{j'} V_{i'}^* = \ker V_{i'}^*$. Therefore, 
\[\bigcap_{i=1}^n \bigcap_{j\geq 0} \ker V_i^* W^j=\bigcap_{i=1}^n \ker V_i^*=\cL_0.\]
Since $\{W,V_1,\cdots,V_n\}$ is a weak bi-shift, we have $W|_{\cL_0}$ is a pure isometry. Therefore, $\cL_0 = \bigoplus_{m\geq 0} W^m (\ker W^* \cap \cL_0) = \bigoplus_{m\geq 0} W^m \cL$. Since $\{V_\mu \cL_0\}_{\mu\in\mathbb{F}_n^+}$ are pairwise orthogonal and $\{ W^m \cL\}_{m\geq 0}$ are pairwise orthogonal, we have $\{V_\mu W^m \cL: \mu\in\mathbb{F}_n^+, m\geq 0\}$ are pairwise orthogonal, and 
\[\cH = \bigoplus_{\mu\in\mathbb{F}_n^+} V_\mu \cL_0 = \bigoplus_{\mu\in\mathbb{F}_n^+} V_\mu (\bigoplus_{m\geq 0} W^m \cL)=\bigoplus_{\mu\in\mathbb{F}_n^+} \bigoplus_{m\geq 0} V_\mu W^m \cL. \qedhere\]
\end{proof}

Finally, after combining these results, we obtain the following Wold-type decomposition for isometric Nica-covariant representations of $\odometer{n}$. 

\begin{theorem}\label{thm.Wold.NC} Let $(W,V_1,V_2,\cdots,V_n)$ be an isometric Nica-covariant representation of the odometer semigroup $\odometer{n}$ on some Hilbert space $\cH$. Then $\cH$ decomposes uniquely as
\[\cH=\cH_{uu}\oplus \cH_{us} \oplus \cH_{su} \oplus \cH_{ss},\]
where,
\begin{enumerate}
    \item The subspaces $\cH_{uu}, \cH_{us}, \cH_{su}, \cH_{ss}$ are reducing for $W$ and $V_i$.
    \item The $W|_{\cH_{uu}}$ is unitary and $\{V_1,\cdots,V_n\}|_{\cH_{uu}}$ is a row unitary.
    \item The $W|_{\cH_{us}}$ is unitary and $\{V_1,\cdots,V_n\}|_{\cH_{us}}$ is pure.
    \item The $W|_{\cH_{su}}$ is pure, and $\{V_1,\cdots,V_n\}|_{\cH_{su}}$ is a row unitary.
    \item The family $\{W, V_1,\cdots,V_n\}|_{\cH_{ss}}$ is unitarily equivalent to a direct sum of copies of the left regular representation of $\odometer{n}$.
\end{enumerate}
\end{theorem}

\begin{proof} Apply Theorem \ref{thm.Wold} to $\{W,V_1,\cdots,V_n\}$ as they are isometric representations of $\odometer{n}$. The first three components are automatically Nica-covariant by Lemma \ref{lm.obs0}. Proposition \ref{prop.NC1} implies that the weak bi-shift component is a direct sum of the left regular representations.
\end{proof} 

\section{Examples} 

This paper is largely motivated by our recent progress on the characterization of atomic Nica-covariant representations of the odometer semigroup. We shall leave the full characterization for a subsequent paper. Atomic representations often provide an simple yet interesting class of representations (see, for example, representations of free semigroup algebra \cite{DP1999}, single vertex $2$-graph \cite{DPY2008}, and free semigroupoid algebras \cite{DDL2019}). One special class of atomic representations is called the permutation representation, first considered for Cuntz algebras in \cite{BJ1999}. 
In this section, we build concrete atomic isometric representations of $\odometer{n}$ for each of the four possible components in the Wold-type decomposition that we established in Theorem \ref{thm.Wold}.

For the odometer semigroup, we say an representation $(W,V_1,\cdots,V_n)$ on a separable Hilbert space $\cH$ is atomic if there is an orthonormal basis $\{e_i\}_{i\in I}$ for $\cH$ and injective maps $\tau$ and $\{\pi_1, \dots, \pi_n\}$ on $I$, such that 
\begin{enumerate}
    \item The ranges of $\pi_k$'s are pairwise disjoint;
    \item For each $i\in I$, $We_i=\lambda_i e_{\tau(i)}$ for some $\lambda_i\in\mathbb{T}$ ;
    \item For each $i\in I$ and $1\leq k\leq n$, $V_k e_i=\omega_{k,i} e_{\pi_k(i)}$ for some $\omega_{k,i}\in\mathbb{T}$. 
\end{enumerate}

Recall that the atomic row isometries $\{V_1,\cdots,V_n\}$ were fully characterized in \cite[Theorem 3.4]{DP1999} to study the free semigroup algebras (see also \cite{DKP2001}). We briefly go over three types of atomic row isometries. 
First, we construct a directed graph whose vertices are $e_i$ and there is an edge from $e_i$ to $e_j$ if there exists $V_k$ such that $V_k e_i\in\lspan\{e_j\}$. Each connected component of this directed graph corresponds to a reducing subspace of the atomic representation, and thus we assume this directed graph is connected. 
For each basic vector $e_{i_0}$, if $e_{i_0}$ is in the range of a unique $V_{k_1}$, we let $e_{i_1}$ be the basic vector such that $V_{k_1} e_{i_1}\in \lspan\{e_{i_0}\}$. Repeat this process, there are three possibilities:
\begin{enumerate}
    \item The process stops at some $e_{i_m}$ that is not in the range of any $V_k$. Then $e_{i_m}$ is a wandering vector, and on $\overline{\lspan}\{V_\mu e_{i_m}: \mu\in\mathbb{F}_n^+\}$, $\{V_1,\cdots,V_n\}$ is unitarily equivalent to the left regular representation of $\mathbb{F}_n^+$. This is called the left-regular type. 
    \item The process enters into a cycle when $e_{i_s}=e_{i_t}$ for some $0\leq s<t$. This corresponds to the cycle type.
    \item The process never stops and we obtain an infinite sequence of unique basic vectors $\{e_{i_m}\}_{m\geq 0}$ and $V_{k_m}$. In this case, we can rescale $e_i$ such that $V_{k_m} e_{i_m}=e_{i_{m-1}}$ for all $m\geq 1$.
    This corresponds to the inductive type. An inductive type atomic representation is called non-cyclic if the sequence $\{k_m\}$ is not eventually periodic (that is, $\{k_m: m\geq M\}$ is non-periodic for each $M$).
\end{enumerate}

We first show that non-cyclic inductive type atomic representations always give rise to a unique unitary-row unitary representation of the odometer semigroup. These are examples for the $\cH_{uu}$ component.

\begin{proposition} Let $\{V_1,\cdots,V_n\}$ be a non-cyclic inductive type atomic representation on $\cH$. Then there exists a unique unitary operator $W$ such that $\{W,V_1,\cdots,V_n\}$ is an atomic representation of the odometer semigroup $\odometer{n}$. Moreover, $W$ is a unitary and $\{V_1,\cdots,V_n\}$ is a row unitary.
\end{proposition}

\begin{proof} First of all inductive type atomic representation gives rise to a row unitary $\{V_1,\cdots,V_n\}$. 
For each basic vector $e_{i_0}\in \cH$, it can be uniquely written as $e_{i_0}=V_{k_1} e_{i_1}$. Since $V$ is non-cyclic, we can eventually find $m\geq 1$ with $k_m\neq n$, so that $e_{i_0}=V_n^{m-1} V_{k_m} e_{i_m}$. Define $We_{i_0} = V_1^{m-1} V_{k_m+1} e_{i_m}$. One can easily verify that $W$ is a well-defined isometry such that $\{W,V_1,\cdots,V_n\}$ is an atomic representation of $\odometer{n}$. To see $W$ is unitary, for any basic vector $e_{i_0}$, we can write it as $e_{i_0}=V_1^{m-1} V_{k_m} e_{i_m}$ for some $2\leq k_m\leq n$, and $e_{i_0}=W V_n^{m-1} V_{k_m-1} e_{i_m}$ by the definition of $W$. 
This choice of $W$ is unique since $WV_n^m V_k=V_1^m V_{k+1}$ for any $m\geq 0$ and $1\leq k\leq n-1$. 
\end{proof}

For unitary-pure row isometry representation, Corollary \ref{cor.us} describes their general structure. It is essentially composed of a pure row isometry and a unitary operator on its wandering space. One can easily build such an atomic representation.

\begin{example} Let $\cH=\overline{\lspan}\{e_\mu: \mu\in\mathbb{F}_n^+\}$ and $V_i$ be the left-regular representation of $\mathbb{F}_n^+$. For each $\lambda\in\mathbb{T}$, define
\[We_\mu = \begin{cases}
e_{(\mu_1+1)\mu_2\cdots \mu_m}, & \text{ if } \mu_1\neq n, \\
e_{1(\mu_2+1)\mu_3\cdots \mu_m}, & \text{ if } \mu_1=n, \mu_2\neq n, \\
& \vdots \\
\lambda e_{\underbrace{1\cdots 1}_{m}}, & \text{ if } \mu_1=\cdots = \mu_m=n. \\
\end{cases}
\]
In particular, for the empty word $\emptyset\in\mathbb{F}_n^+$, by convention, define $We_\emptyset = \lambda e_\emptyset$. 
One can clearly see that $W$ is unitary on $\lspan\{e_\mu: |\mu|=m\}$ for each $m\geq 0$. Therefore, $W$ is unitary. One can verify that $\{W,V_1,\cdots,V_n\}$ is an atomic representation of $\odometer{n}$. Different choices of $\lambda$ clearly induce non-equivalent isometric representations, because $e_\emptyset$ is an eigenvector of $W$ with eigenvalue $\lambda$. Moreover, by Corollary \ref{cor.us}, these are all the unitary-pure row isometry type representations where the row isometry is the left regular representation. 
\end{example}

We now construct a pure-row unitary type that corresponds to the $\cH_{su}$ component. 

\begin{example} Consider the following diagram where each vertex corresponds to a basic vector, dotted arrows correspond to the map $W$, solid arrows corresponds to the map $V_k$ depending on the label. One can verify that this defines an atomic representation $\{W,V_1,V_2\}$ for $\odometer{2}$ where $W$ is a unilateral shift and $\{V_1,V_2\}$ is a row unitary. 

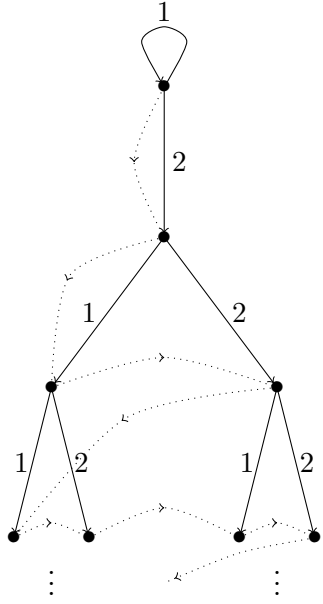
\begin{figure}[h]
    \centering

    \begin{tikzpicture}[scale=1, every node/.style={scale=1}]
    
    \node at (0,4) {$\bullet$};
    
    \node at (0,2) {$\bullet$};
    \node at (-1.5,0) {$\bullet$};
    \node at (1.5,0) {$\bullet$};
    
    \node at (-2,-2) {$\bullet$};
    \node at (-1,-2) {$\bullet$};
    \node at (1,-2) {$\bullet$};
    \node at (2,-2) {$\bullet$};
    
    \draw [decoration={markings,
    mark=at position 0.97 with \arrow{>}},
    postaction=decorate] plot [smooth] coordinates {(0,4) (0.3, 4.6) (0, 4.8) (-0.3,4.6) (0,4)};
    \node at (0,5) {$1$};

	\draw[decoration={markings,
    mark=at position 0.97 with \arrow{>}},
    postaction=decorate] (0,4) -- (0,2);
	\node at (0.2,3) {$2$};
	
	\draw[decoration={markings,
    mark=at position 0.97 with \arrow{>}},
    postaction=decorate] (0,2) -- (-1.5,0);
    \node at (-1,1) {$1$};
	\draw[decoration={markings,
    mark=at position 0.97 with \arrow{>}},
    postaction=decorate] (0,2) -- (1.5,0);
    \node at (1,1) {$2$};
    
    \draw[decoration={markings,
    mark=at position 0.97 with \arrow{>}},
    postaction=decorate] (-1.5,0) -- (-2,-2);
    \node at (-1.9,-1) {$1$};
	\draw[decoration={markings,
    mark=at position 0.97 with \arrow{>}},
    postaction=decorate] (-1.5,0) -- (-1,-2);
    \node at (-1.1,-1) {$2$};
    
    \draw[decoration={markings,
    mark=at position 0.97 with \arrow{>}},
    postaction=decorate] (1.5,0) -- (1,-2);
    \node at (1.1,-1) {$1$};
	\draw[decoration={markings,
    mark=at position 0.97 with \arrow{>}},
    postaction=decorate] (1.5,0) -- (2,-2);
    \node at (1.9,-1) {$2$};

    \draw [dotted, decoration={markings, post length=1mm,
                 pre length=1mm,
    mark=at position 0.5 with \arrow{>}},
    postaction=decorate] plot [smooth] coordinates {(0,4) (-0.4,3) (0,2)};
    
    \draw [dotted, decoration={markings, post length=1mm,
                 pre length=1mm,
    mark=at position 0.5 with \arrow{>}},
    postaction=decorate] plot [smooth] coordinates {(0,2) (-1.25, 1.5) (-1.5,0)};
    
    \draw [dotted, decoration={markings, post length=1mm,
                 pre length=1mm,
    mark=at position 0.5 with \arrow{>}},
    postaction=decorate] plot [smooth] coordinates {(-1.5,0) (0, 0.4) (1.5,0)};
	
    \draw [dotted, decoration={markings, post length=1mm,
                 pre length=1mm,
    mark=at position 0.5 with \arrow{>}},
    postaction=decorate] plot [smooth] coordinates {(1.5,0) (-0.5, -0.4) (-2,-2)};
    
    \draw [dotted, decoration={markings, post length=1mm,
                 pre length=1mm,
    mark=at position 0.5 with \arrow{>}},
    postaction=decorate] plot [smooth] coordinates {(-2,-2) (-1.5, -1.8) (-1,-2)};
    
    \draw [dotted, decoration={markings, post length=1mm,
                 pre length=1mm,
    mark=at position 0.5 with \arrow{>}},
    postaction=decorate] plot [smooth] coordinates {(-1,-2) (0, -1.6) (1,-2)};
    
    \draw [dotted, decoration={markings, post length=1mm,
                 pre length=1mm,
    mark=at position 0.5 with \arrow{>}},
    postaction=decorate] plot [smooth] coordinates {(1,-2) (1.5, -1.8) (2,-2)};
    
    \draw [dotted, decoration={markings, post length=1mm,
                 pre length=1mm,
    mark=at position 0.97 with \arrow{>}},
    postaction=decorate] plot [smooth] coordinates {(2,-2) (1, -2.2) (0,-2.6)};
    
    \node at (-1.5,-2.5) {$\vdots$};
    \node at (1.5,-2.5) {$\vdots$};
    \end{tikzpicture}
    \label{fig:B.sB}
\end{figure}
\end{example}

Finally, for the weak bi-shift component, the left-regular representation of $\odometer{n}$ clearly falls under this category. We would like to construct an example that is not from the left regular representation of $\odometer{n}$. The atomic representation in the following example is in the same spirit of S\l oci\'{n}ski's example, and in fact, this research is mostly motivated by their similarity. 

\begin{example}\label{ex.ws}  Consider the atomic representation $\{W,V_1,V_2\}$ defined by the following diagram.
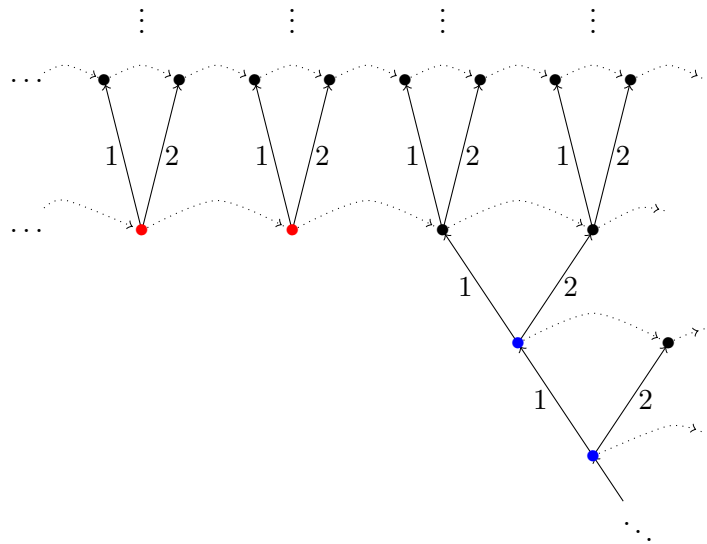
\begin{figure}[h]
    \centering

    \begin{tikzpicture}[scale=1, every node/.style={scale=1}]
    
    \node at (0,0) {$\bullet$};
   
    \node at (2,0) {$\bullet$};
    
    \node at (-2.5,2) {$\bullet$};
    \node at (-1.5,2) {$\bullet$};
    \node at (-4.5,2) {$\bullet$};
    \node at (-3.5,2) {$\bullet$};
    
    \node at (0.5,2) {$\bullet$};
    \node at (-0.5,2) {$\bullet$};
    \node at (1.5,2) {$\bullet$};
    \node at (2.5,2) {$\bullet$};

    \node at (3,-1.5) {$\bullet$};
    
    \node at (-5.5,2) {$\dots$};
    \node at (-5.5,0) {$\dots$};
    
    \foreach \x in {0,...,7}{
        \draw [dotted, decoration={markings, post length=1mm,
                 pre length=1mm,
    mark=at position 0.97 with \arrow{>}},
    postaction=decorate] plot [smooth] coordinates {(-4.5+\x, 2) (-4+\x,2.2) (-3.5+\x,2)};
    }
    \foreach \x in {0,...,3}{
        \draw[decoration={markings,
    mark=at position 0.97 with \arrow{>}},
    postaction=decorate] (-4+2*\x, 0 ) -- (-4.5+2*\x,2);
    \node at (-4.4+2*\x,1) {$1$};
    \draw[decoration={markings,
    mark=at position 0.97 with \arrow{>}},
    postaction=decorate] (-4+2*\x, 0 ) -- (-3.5+2*\x,2);
    \node at (-3.6+2*\x,1) {$2$};
    \node at (-4+2*\x,2.9) {$\vdots$};
    }
    
    \draw[decoration={markings,
    mark=at position 0.97 with \arrow{>}},
    postaction=decorate] (2+0.4, -3-0.4*1.5) -- (2,-3);
    \node at (2+0.6, -3-0.6*1.5) {$\ddots$};
    
    \draw [dotted, decoration={markings, post length=1mm,
                 pre length=1mm,
    mark=at position 0.97 with \arrow{>}},
    postaction=decorate] plot [smooth] coordinates {(-5.3, 2.1) (-5,2.2) (-4.5,2)};
    \draw [dotted, decoration={markings, post length=1mm,
                 pre length=1mm,
    mark=at position 0.97 with \arrow{>}},
    postaction=decorate] plot [smooth] coordinates {(-5.3, 0.3) (-5,0.4) (-4,0)};
    
    \draw [dotted, decoration={markings, post length=1mm,
                 pre length=1mm,
    mark=at position 0.97 with \arrow{>}},
    postaction=decorate] plot [smooth] coordinates {(-4, 0) (-3,0.4) (-2,0)};
    \draw [dotted, decoration={markings, post length=1mm,
                 pre length=1mm,
    mark=at position 0.97 with \arrow{>}},
    postaction=decorate] plot [smooth] coordinates {(-2, 0) (-1,0.4) (-0,0)};
    \draw [dotted, decoration={markings, post length=1mm,
                 pre length=1mm,
    mark=at position 0.97 with \arrow{>}},
    postaction=decorate] plot [smooth] coordinates {(0, 0) (1,0.4) (2,0)};
    \draw [dotted, decoration={markings, post length=1mm,
                 pre length=1mm,
    mark=at position 0.97 with \arrow{>}},
    postaction=decorate] plot [smooth] coordinates {(2, 0) (2.6,0.3) (3, 0.25)};
    
    \draw [dotted, decoration={markings, post length=1mm,
                 pre length=1mm,
    mark=at position 0.97 with \arrow{>}},
    postaction=decorate] plot [smooth] coordinates {(1, -1.5) (2, -1.1) (3, -1.5)};
    
    \draw [dotted, decoration={markings, post length=1mm,
                 pre length=1mm,
    mark=at position 0.97 with \arrow{>}},
    postaction=decorate] plot [smooth] coordinates {(3, -1.5) (3.2, -1.4) (3.5, -1.3)};
    
    \draw [dotted, decoration={markings, post length=1mm,
                 pre length=1mm,
    mark=at position 0.97 with \arrow{>}},
    postaction=decorate] plot [smooth] coordinates {(2, -3) (3, -2.6) (3.5, -2.7)};

    \draw[decoration={markings,
    mark=at position 0.97 with \arrow{>}},
    postaction=decorate] (1,-1.5) -- (0,0);
    \node at (0.3, -0.75) {$1$};
    
    \draw[decoration={markings,
    mark=at position 0.97 with \arrow{>}},
    postaction=decorate] (2,-3) -- (1,-1.5);
    \node at (1.3, -0.75-1.5) {$1$};
    
    \draw[decoration={markings,
    mark=at position 0.97 with \arrow{>}},
    postaction=decorate] (2,-3) -- (3,-1.5);
    \node at (2.7, -0.75-1.5) {$2$};
    \draw[decoration={markings,
    mark=at position 0.97 with \arrow{>}},
    postaction=decorate] (1,-1.5) -- (2,0);
    \node at (1.7, -0.75) {$2$};
    
     \node [red] at (-2,0) {$\bullet$};
    \node [red] at (-4,0) {$\bullet$};
    \node [blue] at (1,-1.5) {$\bullet$};
    \node [blue] at (2,-3) {$\bullet$};
    
    \end{tikzpicture}
    \label{fig:B.sD}
\end{figure}

 One can easily verify that this atomic representation has no $\cH_{uu}$, $\cH_{us}$, and $\cH_{su}$ components, and it is not the left-regular representation of $\odometer{2}$. 

One may notice that the red dots are wandering vectors for $\{V_1,V_2\}$ while the blue dots are wandering vectors for $W$. 
This is quite similar to the S\l oci\'{n}ski's example \cite[Example 1]{Slocinski1980} in which the vectors $\{e_{0,-n}\}_{n\geq 1}$ are wandering for $S_1$ and $\{e_{-n,0}\}_{n\geq 1}$ are wandering for $S_2$. 
\end{example}

\bibliographystyle{abbrv}
%\bibliography{semigroup}

\end{document}